\documentclass[11pt]{amsart}
\usepackage{amssymb}
\usepackage{amsmath}
\usepackage[active]{srcltx}
\usepackage{t1enc}
\usepackage[latin2]{inputenc}
\usepackage{verbatim}
\usepackage{amsmath,amsfonts,amssymb,amsthm}
\usepackage[mathcal]{eucal}
\usepackage{enumerate}
\usepackage[centertags]{amsmath}
\usepackage{graphics}

\newtheorem{theorem}{Theorem}

\newtheorem{lemma}{Lemma}
\newtheorem{remark}{Remark}

\newtheorem{problem}{Problem}
\newtheorem{proposition}{Proposition}

\begin{document}
\author{D. Lukkassen, L.E. Persson, G. Tephnadze, G. Tutberidze}
\title[Reisz logarithmic means ]{Some inequalities related to strong convergence of Riesz
logarithmic means}

\address{D. Lukkassen, UiT The Arctic University of Norway, P.O. Box 385, N-8505, Narvik, Norway}
\email{Dag.Lukkassen@uit.no}

\address{L.E. Persson, UiT The Arctic University of Norway, P.O. Box 385, N-8505, Narvik, Norway and Department of Mathematics and Computer Science, Karlstad University, 65188 Karlstad, Sweden.}
\email{lars.e.persson@uit.no}

\address{G. Tephnadze, The University of Georgia, School of Science and Technology, 77a Merab Kostava St, Tbilisi, 0128, Georgia}
\email{g.tephnadze@ug.edu.ge }

\address{G. Tutberidze, The University of Georgia, School of IT, Engineering and Mathematics, 77a Merab Kostava St, Tbilisi, 0128, Georgia and UiT The Arctic University of Norway, P.O. Box 385, N-8505, Narvik, Norway}
\email{giorgi.tutberidze1991@gmail.com}

\thanks{The research was supported by Shota Rustaveli National Science Foundation grant YS-18-043.}
\date{}
\maketitle

\begin{abstract}
In this paper we derive a new strong convergence theorem of Riesz logarithmic means of the one-dimensional Vilenkin-Fourier (Walsh-Fourier) series. The corresponding inequality is pointed out and it is also proved that the inequality is in a sense sharp, at least for the case with Walsh-Fourier series.
\end{abstract}

\date{}



\section{Introduction}

Concerning definitions used in this introduction we refer to section 2. Weisz \cite{We2} proved the boundedness of the maximal operator of Fejér means  $\sigma ^{\psi,\ast }$ with respect to bounded Vilenkin systems from the martingale Hardy space $H_p(G_m)$ to the space $L_p(G_m),$ for $p>1/2$. Simon \cite{Si1} gave a counterexample, which shows that boundedness does not hold for $0<p<1/2.$ The corresponding counterexample for $p=1/2$ is due to Goginava \cite{Go}. Moreover, Weisz \cite{We3} proved the following result:

\textbf{Theorem W. }The maximal operator of Fejér means $\sigma ^{\psi,\ast }$ is bounded from the Hardy space $H_{1/2}(G_m)$ to the space $weak-L_{1/2}(G_m).$

In \cite{tep2} and \cite{tep3} it were proved that the maximal operator $
\widetilde{\sigma }_{p}^{\psi,\ast}$ defined by
\begin{equation*}
\widetilde{\sigma}_{p}^{\psi,\ast}:=\sup_{n\in \mathbb{N}}\frac{\left\vert\sigma_{n}^{\psi}\right\vert }{\left( n+1\right)^{1/p-2}\log^{2\left[ 1/2+p\right]}\left(n+1\right) },
\end{equation*}
where $0<p\leq 1/2$ and $\left[ 1/2+p\right] $ denotes integer part of $%
1/2+p,$ is bounded from the Hardy space $H_p(G_m)$ to the space $L_p(G_m).$ Moreover, for any nondecreasing function $\varphi :\mathbb{N}_{+}\rightarrow\lbrack 1,\infty)$ satisfying the condition
\begin{equation}\label{cond1}
\overline{\lim_{n\rightarrow \infty }}\frac{\left( n+1\right) ^{1/p-2}\log
	^{2\left[ 1/2+p\right] }\left( n+1\right) }{\varphi \left( n\right) }%
=+\infty ,  
\end{equation}%
there exists a martingale $f\in H_p(G_m),$ such that
\begin{equation*}
\underset{n\in \mathbb{N}}{\sup}\left\Vert \frac{\sigma^\psi_nf}{\varphi
	\left( n\right) }\right\Vert _{p}=\infty .
\end{equation*}

For Walsh-Kaczmarz system some analogical results were proved in \cite{GNCz} and \cite{tep5}.

Weisz \cite{We2} considered the norm convergence of Fej\'er means of
Vilenkin-Fourier series and proved the following result:

\textbf{Theorem W1\ (Weisz).} Let $p>1/2$ and $f\in H_p(G_m).$ Then there
exists an absolute constant $c_{p}$, depending only on $p$, such that for
all $k=1,2,\dots$ and $f\in H_p(G_m)$ the following inequality holds:
\begin{equation*}
\left\Vert \sigma^{\psi}_kf\right\Vert_p\leq c_p\left\Vert f\right\Vert
_{H_p(G_m)}.
\end{equation*}

Moreover, in \cite{tep1} it was proved that the assumption $p>1/2$ in Theorem W1 is essential. In fact, the following is true:

\textbf{Theorem T1.} There exists a martingale $f\in H_{1/2}(G_m)$ such that
\begin{equation*}
\sup_{n\in\mathbb{N}}\left\Vert \sigma^{\psi}_nf\right\Vert _{1/2}=+\infty.
\end{equation*}

Theorem W1 implies that
\begin{equation*}
\frac{1}{n^{2p-1}}\overset{n}{\underset{k=1}{\sum }}\frac{\left\Vert \sigma^{\psi}_k f\right\Vert_p^p}{k^{2-2p}}\leq c_p\left\Vert f\right\Vert
_{H_p(G_m)}^p,\text{ \ \ \ }1/2<p<\infty,\ n=1,2,\dots.
\end{equation*}

If Theorem W1 holds for $0<p\leq 1/2,$ then we would have that
\begin{equation} \label{2cc}
\frac{1}{\log ^{\left[ 1/2+p\right] }n}\overset{n}{\underset{k=1}{\sum }}\frac{\left\Vert \sigma^\psi_kf\right\Vert_p^p}{k^{2-2p}}\leq
c_{p}\left\Vert f\right\Vert _{H_p(G_m)}^p, \ 0<p\leq 1/2, \ \ n=2,3,\dots.  
\end{equation}

For the Walsh system in \cite{tep6} and for the bounded Vilenkin systems in \cite{tep5} were proved that (\ref{2cc}) holds, though Theorem T1 is not true for $0<p<1/2.$

Some results concerning summability of Fejér means of Vilenkin-Fourier series can be
found in \cite{Fu,Gat,GNCz,PS,Sc,Si2}.

Riesz logarithmic means with respect to the Walsh system was studied by
Simon \cite{Si1}, Goginava \cite{gogin}, Gát, Nagy \cite{GN} and for
Vilenkin systems by Gát \cite{Ga1} and Blahota, Gát \cite{bg}. Moreover, in \cite{PTW1} it was proved that the maximal operator of Riesz
logarithmic means of Vilenkin-Fourier series is bounded from the martingale
Hardy space $H_p(G_m)$ to the space $L_p(G_m)$ when $p>1/2$ and is not bounded
from the martingale Hardy space $H_p(G_m)$ to the space $L_p(G_m)$ when $0<p\leq1/2.$

In \cite{tep2} and \cite{tep3} it were proved that Riesz logarithmic means have better properties than Fejér means. In particular, it was considered the following maximal operator $\widetilde{R}_p^{\psi,\ast}$ of Riesz logarithmic means
$\widetilde{R}_{p}^{\psi,\ast}$ defined by
\begin{equation*}
\widetilde{R}_p^{\psi,\ast}:=\sup_{n\in \mathbb{N}}\frac{\left\vert R_{n}^{\psi}\right\vert \log(n+1)}{\left(n+1\right) ^{1/p-2}\log ^{2\left[ 1/2+p\right] }\left( n+1\right)},
\end{equation*}
where $0<p\leq 1/2$ and $\left[ 1/2+p\right] $ denotes integer part of $
1/2+p,$ is bounded from the Hardy space $H_p(G_m)$ to the space $L_p(G_m).$

Moreover, this result is sharp in the following sense: For any nondecreasing function $\varphi :\mathbb{N}_{+}\rightarrow\lbrack 1,\infty)$ satisfying the condition
\begin{equation}\label{cond2}
\overline{\lim_{n\rightarrow \infty}}\frac{\left(n+1\right)^{1/p-2}\log^{2\left[1/2+p\right]} \left(n+1\right)}{\varphi\left(n\right)\log(n+1)}=\infty,
\end{equation}
there exists a martingale $f\in H_p(G_m),$ such that
\begin{equation*}
\underset{n\in\mathbb{N}}{\sup}\left\Vert\frac{R^{\psi}_nf}{\varphi\left(n\right)} \right\Vert_p=\infty.
\end{equation*}

The main aim of this paper is to derive a new strong convergence theorem of Riesz logarithmic means of one-dimensional Vilenkin-Fourier (Walsh-Fourier)
series (see Theorem 1). The corresponding inequality is pointed out. The sharpness is proved in Theorem 2, at least for the case with Walsh-Fourier series. 

The paper is organized as follows: In Section 2 some definitions and notations are presented. The main results are presented and proved in Section 3.  Section 4 is reserved for some concluding remarks and open problems.

\section{Definitions and Notations}

Let $\mathbb{N}_{+}$ denote the set of positive integers, $\mathbb{N}:=\mathbb{N}_{+}\cup \{0\}.$

Let $m:=(m_0, m_1,...)$ denote a sequence of positive integers not
less than 2.

Denote by
\begin{equation*}
Z_{m_{k}}:=\{0,1,...m_{k}-1\}
\end{equation*}
the additive group of integers modulo $m_{k}.$

Define the group $G_m$ as the complete direct product of the group $Z_{m_j}$ with the product of the discrete topologies of $Z_{m_j}$'s.

The direct product $\mu $ of the measures
\begin{equation*}
\mu_k\left( \{j\}\right):=1/m_{k}\text{\qquad}(j\in Z_{m_{k}})
\end{equation*}%
is a Haar measure on $G_m$ with $\mu\left(G_m\right)=1.$

If $\sup_{n\in \mathbb{N}}m_n<\infty$, then we call $G_{m}$ a bounded Vilenkin group.
If the generating sequence $m$ is not bounded, then $G_{m}$ is said to be an
unbounded Vilenkin group. In this paper we discuss only bounded Vilenkin groups.

The elements of $G_m$ are represented by sequences
\begin{equation*}
x:=(x_0,x_1,...,x_j,...)\qquad \left( \text{ }x_k\in
Z_{m_k}\right).
\end{equation*}

It is easy to give a base for the neighborhood of $G_{m}$ namely
\begin{equation*}
I_{0}\left( x\right) :=G_{m},
\end{equation*}%
\begin{equation*}
I_{n}(x):=\{y\in G_{m}\mid y_{0}=x_{0},...y_{n-1}=x_{n-1}\},\text{ \ }(x\in
G_{m},\text{ }n\in \mathbb{N}).
\end{equation*}

Denote $I_{n}:=I_{n}\left( 0\right) $ for $n\in \mathbb{N}$ and $\overline{%
I_{n}}:=G_{m}$ $\backslash $ $I_{n}$.

Let
\begin{equation*}
e_{n}:=\left( 0,0,...,x_{n}=1,0,...\right) \in G_{m}\qquad \left( n\in\mathbb{N}\right).
\end{equation*}

It is evident that
\begin{equation}\label{3}
\overline{I_M}=\left(\overset{M-2}{\underset{k=0}{\bigcup}}\overset{m_k-1} {\underset{x_k=1}{\bigcup}}\overset{M-1}{\underset{l=k+1}{\bigcup}} \overset{m_l-1}{\underset{x_l=1}{\bigcup}}I_{l+1}\left(x_ke_k+x_le_l\right) \right)\bigcup\left( \underset{k=1}{\bigcup\limits^{M-1}} \overset{m_{k}-1} {\underset{x_k=1}{\bigcup}}I_M\left(x_ke_k\right)\right). 
\end{equation}

If we define the so-called generalized number system based on $m$ in the
following way :
\begin{equation*}
M_{0}:=1,\text{ \qquad }M_{k+1}:=m_{k}M_{k\text{ }}\ \qquad (k\in \mathbb{N}),
\end{equation*}
then every $n\in \mathbb{N}$ can be uniquely expressed as $n=\overset{\infty
}{\underset{k=0}{\sum }}n_{j}M_{j},$ where $n_{j}\in Z_{m_{j}}$ $~(j\in
\mathbb{N})$ and only a finite number of $n_{j}`$s differ from zero. Let $\left| n\right|:=\max \{j\in \mathbb{N}; \ \ n_j\neq 0\}.$

The norm (or quasi-norm when $p<1$) of the space $L_{p}(G_{m})$ is defined by \qquad
\qquad \thinspace \
\begin{equation*}
\left\| f\right\| _{p}:=\left( \int_{G_{m}}\left| f\right| ^{p}d\mu \right)
^{1/p}\qquad \left( 0<p<\infty \right) .
\end{equation*}

The space $weak-L_p\left( G_m\right) $ consists of all measurable
functions $f$ for which

\begin{equation*}
\left\Vert f\right\Vert_{weak-L_p(G_m)}:=\underset{\lambda>0}{\sup} \lambda^p \mu\left(f>\lambda \right)<+\infty.
\end{equation*}

Next, we introduce on $G_m$ an orthonormal system which is called the
Vilenkin system.

Let us define the complex valued function $r_{k}\left( x\right)
:G_{m}\rightarrow\mathbb{C},$ the generalized Rademacher functions as
\begin{equation*}
r_k\left( x\right) :=\exp \left( 2\pi ix_{k}/m_{k}\right) \text{ \qquad }\left( i^{2}=-1,\ \ \ x\in G_{m},\ \ \ k\in \mathbb{N}\right) .
\end{equation*}

Now, define the Vilenkin system $\psi :=(\psi _{n}:n\in \mathbb{N})$ on $%
G_{m}$ as:
\begin{equation*}
\psi_n(x):=\overset{\infty}{\underset{k=0}{\Pi}}r_k^{n_k}\left(
x\right)\text{ \qquad }\left(n\in \mathbb{N}\right).
\end{equation*}

The Vilenkin systems are orthonormal and complete in $L_{2}\left( G_{m}\right)$ (for details see e.g \cite{AVD}).

Specifically, we call this system the Walsh-Paley one 
if $m_k= 2, \ \text{for all} \ k\in \mathbb{N}.$ In this case we have dyadic group $G_2=\prod_{j=0}^{\infty}Z_{2},$ which is called the Walsh group and Vilenkin system coincides Walsh functions defined by (for details see e.g. \cite{G-E-S} and \cite{S-W-S})
\begin{equation*}
w_{n}(x):=\underset{k=0}{\overset{\infty }{\prod }}r_{k}^{n_{k}}\left(
x\right) =r_{\left\vert n\right\vert }\left( x\right) \left( -1\right) ^{\underset{k=0}{\overset{\left\vert n\right\vert -1}{\sum }}n_{k}x_{k}}\text{\qquad }\left( n\in \mathbb{N}\right),
\end{equation*}
where $n_k=0\vee1$ and $x_k=0\vee 1.$

Now, we introduce analogues of the usual definitions in Fourier-analysis.

If $f\in L_{1}\left( G_{m}\right), $ then we can establish the Fourier
coefficients, the partial sums of the Fourier series, the Fejér means, the
Dirichlet and Fejér kernels with respect to the Vilenkin system $\psi $ (Walsh system $w$) in
the usual manner:%
\begin{eqnarray*}
\widehat{f}^{\alpha }\left( k\right) &:&=\int_{G_m}f\overline{\alpha }_kd\mu,\text{\ \ \ \ \ \ } (\alpha_k=w_{k}\text{ or }\psi_k), \qquad \left(k\in \mathbb{N}\right) , \\
S^\alpha_n f&:&=\sum_{k=0}^{n-1}\widehat{f}\left(k\right)\alpha_k, \qquad  (\alpha _{k}=w_{k}\text{ or }\psi_k),\text{ \ \
\ }\left(n\in \mathbb{N}_{+}, \ S^{\alpha}_0 f:=0\right) , \\
\sigma^{\alpha}_n f &:&=\frac{1}{n}\sum_{k=0}^{n-1}S^{\alpha}_k f,\text{ \ \ \ \ \ \ } (\alpha=w\text{ or }\psi ), \ \ \ \ \ \ \ \ \ \ \left(n\in \mathbb{N}_{+}\right) , \\
D^{\alpha}_n &:&=\sum_{k=0}^{n-1}\alpha_{k},\text{ \qquad\ \ \ \ \ } (\alpha=w\text{ or }\psi ), \ \ \ \ \ \ \ \ \ \  \left(n\in \mathbb{N}_{+}\right), \\
K^{\alpha}_n &:&=\frac{1}{n}\overset{n-1}{\underset{k=0}{\sum }}D^{\alpha}_{k}, \ \ \ \ \ \ \ \ \ (\alpha=w\text{ or }\psi ),  \text{ \qquad \ \ }\left(n\in \mathbb{N}_{+}\right) .
\end{eqnarray*}%
$\qquad $ $\qquad $

It is well-known that (see e.g. \cite{AVD})

\begin{equation} \label{kodala}
\sup_{n\in \mathbb{N}}\int_{G_m}\left\vert K^{\alpha}_n\right\vert d\mu \leq c<\infty, \ \ \ \ \text{where}\ \ \ \ \alpha=w\text{ or }\psi.
\end{equation}

The $\sigma $-algebra generated by the intervals $\left\{ I_{n}\left(
x\right) :x\in G_{m}\right\} $ will be denoted by $\digamma _{n}$ $\left(
n\in \mathbb{N}\right) .$ Denote by $f=\left( f^{\left( n\right) },n\in
\mathbb{N}\right) $ a martingale with respect to $\digamma _{n}$ $\left(
n\in \mathbb{N}\right)$ (for details see e.g. \cite{cw}, \cite{Ne}, \cite{We1}). The maximal
function of a martingale $f$ is defend by \qquad
\begin{equation*}
f^{\ast }=\sup_{n\in \mathbb{N}}\left\vert f^{\left( n\right) }\right\vert.
\end{equation*}

In the case $f\in L_1(G_m),$ the maximal functions are also given by
\begin{equation*}
f^{\ast }\left( x\right) =\sup_{n\in \mathbb{N}}\frac{1}{\left\vert
I_{n}\left( x\right) \right\vert }\left\vert \int_{I_{n}\left( x\right)
}f\left( u\right) \mu \left( u\right) \right\vert .
\end{equation*}

For $0<p<\infty$ the Hardy martingale spaces $H_p\left(G_m\right)$
consist of all martingales for which
\begin{equation*}
\left\Vert f\right\Vert _{H_p(G_m)}:=\left\Vert f^{\ast }\right\Vert
_p<\infty.
\end{equation*}

If $f\in L_{1}(G_m),$ then it is easy to show that the sequence $\left(
S_{M_n}f :n\in \mathbb{N}\right) $ is a martingale. If $
f=\left( f^{\left( n\right) },n\in \mathbb{N}\right) $ is a martingale, then the Vilenkin-Fourier (Walsh-Fourier) coefficients must be defined in a slightly different
manner, namely
\begin{equation*}
\widehat{f}\left( i\right):=\lim_{k\rightarrow \infty}\int_{G_m} f^{\left( k\right)}\left( x\right) \overline{\alpha}_{i}\left(x\right) d\mu \left( x\right), \ \ \ \text{where} \ \ \ \alpha=w\text{ \ or \ }\psi.
\end{equation*}%
\qquad \qquad \qquad \qquad

The Vilenkin-Fourier coefficients of $f\in L_{1}\left( G_{m}\right) $ are
the same as those of the martingale 
$\left( S_{M_n}f:n\in\mathbb{N}\right)$ obtained from $f.$

In the literature, there is the notion of Riesz logarithmic means of the
Fourier series. The $n$-th Riesz logarithmic means of the Fourier series
of an integrable function $f$ is defined by

\begin{equation*}
R^\alpha_nf:=\frac{1}{l_{n}}\overset{n}{\underset{k=1}{\sum}} \frac{S^\alpha_k f}{k},\ \ \ \text{where}\ \ \ \alpha=w \ \text{or }\psi,
\end{equation*}
with  
$$l_n:=\sum_{k=1}^n\frac{1}{k}.$$ 

The kernels of Riesz`s logarithmic means are defined by

\begin{equation*}
L^\alpha_n:=\frac{1}{l_{n}}\overset{n}{\underset{k=1}{\sum }}\frac{D^\alpha_k}{k}, \ \ \ \text{where} \ \ \  (\alpha =w\text{ or }\psi ).
\end{equation*}

For the martingale $f$ we consider the following maximal operators:
\begin{eqnarray*}
&&\sigma ^{\alpha,\ast }f:\sup_{n\in \mathbb{N}}\left\vert \sigma^{\alpha}
_{n}f\right\vert ,\ \ \ \ \ \ \ \ \ \ \ \ \ \ \ \ \ \ \ \ \ \ \ \ \ (\alpha =w\text{ or }\psi), \\
&&R^{\ast }f:=\underset{n\in \mathbb{N}}{\sup }\left\vert R^{\alpha}_n f\right\vert,\ \ \ \ \ \ \ \ \ \ \ \ \ \ \ \ \ \ \ \ \ \ \ \ 
(\alpha =w\text{ or }\psi),\\
&&\overset{\sim }{R}^{\alpha,\ast }f :=\underset{n\in \mathbb{N}}{\sup }\frac{\left\vert R^{\alpha}_{n}f\right\vert }{\log\left( n+1\right)}
,\ \ \ \ \ \ \ \ \ \ \ \ \ \ \ \ (\alpha =w\text{ or }\psi), \\
&&\overset{\sim }{R}_{p}^{\alpha,\ast }f:=\underset{n\in \mathbb{N}}{\sup }\frac{\log \left(
n+1\right)\left\vert R^{\alpha}_{n}f\right\vert }{\left( n+1\right) ^{1/p-2}}, \ \ \ \ \ \ \ \ (\alpha =w\text{ or }\psi).
\end{eqnarray*}
A bounded measurable function $a$ is a $p$-atom, if there exist an
interval $I$, such that \qquad
\begin{equation*}
\int_{I}ad\mu =0,\text{ \ }\left\Vert a\right\Vert _{\infty }\leq \mu \left(
I\right) ^{-1/p},\text{\ \ supp}\left( a\right) \subset I.
\end{equation*}

\bigskip In order to prove our main results we need the following lemma of Weisz (for details see e.g. Weisz \cite{We4}):

\begin{proposition}\label{atomicdec} A martingale $f=\left( f^{\left( n\right) },n\in \mathbb{N}\right) $ is in $H_p(G_m)\left( 0<p\leq 1\right) $ if and only if there exist
a sequence $\left( a_{k},k\in \mathbb{N}\right) $ of p-atoms and a sequence $%
\left( \mu _{k},k\in \mathbb{N}\right) $ of a real numbers such that for
every $n\in \mathbb{N}$

\begin{equation}
\qquad \sum_{k=0}^{\infty }\mu _{k}S_{M_{n}}a_{k}=f^{\left( n\right)},
\label{6.1}
\end{equation}
and
\begin{equation*}
\qquad \sum_{k=0}^{\infty }\left\vert \mu _{k}\right\vert ^{p}<\infty .
\end{equation*}%
\textit{Moreover, }$\left\Vert f\right\Vert _{H_p(G_m)}\backsim \inf \left(
\sum_{k=0}^{\infty }\left\vert \mu_k\right\vert ^{p}\right) ^{1/p},$ where the infimum is taken over all decompositions of $f$ 
of the form (\ref{6.1}).
\end{proposition}

By using atomic characterization (see Proposition \ref{atomicdec}) it can be easily proved that the following
statement holds (see e.g. Weisz  \cite{We3}):

\begin{proposition}\label{lemma2.2} Suppose that an operator $T$ is sub-linear and for some $0<p_0\leq 1$
\begin{equation*}
\int\limits_{\overset{-}{I}}\left\vert Ta\right\vert ^{p_0}d\mu \leq
c_{p}<\infty
\end{equation*}%
for every $p_0$-atom  $a$, where $I$ denotes the support of the atom. If $T$ is bounded from $L_{p_1}$  to $L_{p_1},$ $\left(1<p_{1}\leq \infty\right) $, then
\begin{equation} \label{p_0}
\left\Vert Tf\right\Vert _{p_0}\leq c_{p_0}\left\Vert f\right\Vert _{H_{p_0}(G_m)}.
\end{equation}
\end{proposition}

The definition of classical Hardy spaces and real Hardy spaces and related theorems of atomic decompositions of these spaces can be found e.g. in Fefferman and Stein \cite{FS} (see also Later \cite{La}, Torchinsky \cite{Tor1}, Wilson \cite{Wil}).

\section{Main Results}

Our first main result reads:
\begin{theorem} \label{threisz_2} Let $0<p<1/2$ and $f\in H_p(G_m).$ Then there exists an absolute constant $c_{p},$ depending only on $p,$ such that the inequality
\begin{equation} \label{star}
\overset{\infty}{\underset{n=1}{\sum}}
\frac{\log^p n \left\Vert R^\psi_nf\right\Vert _{H_p(G_m)}^p}{n^{2-2p}}\leq
c_p\left\Vert f\right\Vert_{H_p(G_m)}^p
\end{equation}
holds, where $R^\psi_{n}f$ denotes the $n$-th Reisz logarithmic mean with respect to the Vilenkin-Fourier series of $f.$ 
\end{theorem}

For the proof of Theorem \ref{threisz_2} we will use the following lemmas:

\begin{lemma} \label{l1} (see \cite{tep6}) \ \ Let $x\in I_{N}\left(x_ke_k+x_le_l\right),$ \ \ $\ 1\leq x_{k}\leq m_{k}-1,$ \ \ $1\leq x_{l}\leq m_{l}-1,$ \  $k=0,...,N-2,$ $\ l=k+1,...,N-1.$ Then	
\begin{equation*}
\int_{I_{N}}\left| K^\psi_{n}\left( x-t\right) \right| d\mu \left( t\right) \leq\frac{cM_{l}M_{k}}{nM_{N}},\qquad \text{when }n\geq M_{N}.
\end{equation*}
Let $x\in I_{N}\left( x_{k}e_{k}\right) , \ \ 1\leq x_{k}\leq m_{k}-1,$ $k=0,...,N-1.$ Then
\begin{equation*}
\int_{I_{N}}\left| K^\psi_{n}\left( x-t\right) \right| d\mu \left( t\right) \leq\frac{cM_{k}}{M_{N}},\qquad \text{when }n\geq M_{N}.
\end{equation*}
\end{lemma}

\begin{lemma} \label{l2} (see \cite{tep7})
Let $x\in I_{N}\left( x_{k}e_{k}+x_{l}e_{l}\right) ,$ $\ 1\leq x_{k}\leq
m_{k}-1,$ $1\leq x_{l}\leq m_{l}-1,$ $\ k=0,...,N-2,$ $l=k+1,...,N-1.$ Then
\begin{equation*}
\int_{I_{N}}\text{ }\underset{j=M_{N}+1}{\overset{n}{\sum }}\frac{\left|
K^\psi_j\left( x-t\right) \right| }{j+1}d\mu \left( t\right) \leq \frac{
cM_{k}M_{l}}{M_{N}^{2}}.
\end{equation*}
	
Let $x\in I_{N}\left( x_{k}e_{k}\right) ,$ $1\leq x_{k}\leq m_{k}-1,$ $\
k=0,...,N-1.$ Then
\begin{equation*}
\int_{I_{N}}\underset{j=M_{N}+1}{\overset{n}{\sum }}\frac{\left| K^\psi_{j}\left(x-t\right) \right| }{j+1}d\mu \left( t\right) \leq \frac{cM_{k}}{M_{N}}l_{n}.
	\end{equation*}
\end{lemma}

\begin{proof} By using Abel transformation, the kernels of the Riesz
logarithmic means can be rewritten as (see also \cite{tep7})
\begin{equation}
L^\psi_n=\frac{1}{l_{n}}\overset{n-1}{\underset{j=1}{\sum }}\frac{K^\psi_j}{j+1}+\frac{K^\psi_n}{l_{n}}.  \label{25}
\end{equation}

Hence, according to \eqref{kodala} we get that
\begin{equation*}
\sup_{n\in \mathbb{N}}\int_{G_m}\left\vert L^{\alpha}_n\right\vert d\mu \leq c<\infty, \ \ \ \ \text{where}\ \ \ \ \alpha=w\text{ or }\psi
\end{equation*}
and it follows that $R^\psi_n$ is bounded from $L_{\infty}$  to $L_{\infty}.$
By Proposition \ref{lemma2.2}, the proof of Theorem \ref{threisz_2} will be complete, if we show that
\begin{equation} \label{reiszbounded}
\overset{\infty}{\underset{n=1}{\sum}}\frac{\log^pn \int\limits_{\overset{-}{I}}\left\vert R^\psi_na\right\vert^pd\mu}{n^{2-2p}}\leq c_p<\infty,\ \ \text{for} \ \  0<p<1/2,
\end{equation}
for every $p$-atom $a,$ where $I$ denotes the support of the atom.

Let $a$ be an arbitrary p-atom with support $I$ and $\mu \left( I\right)
=M_{N}^{-1}.$ We may assume that $I=I_{N}.$ It is easy to see that $%
R^\psi_n a=\sigma^\psi_n\left(a\right)=0,$ when $n\leq M_{N}$.
Therefore we suppose that $n>M_{N}.$

Since $\left\Vert a\right\Vert _{\infty }\leq cM_{N}^{2}$ if we apply (\ref{25}), then we can conclude that
\begin{eqnarray} \label{26}
&&\left\vert R^\psi_na\left(x\right)\right\vert \\ \notag
&=&\int_{I_{N}}\left\vert a\left(t\right)\right\vert\left\vert L^\psi_n\left(x-t\right)\right\vert d\mu(t)   \\
&\leq &\left\Vert a\right\Vert _{\infty}\int_{I_{N}}\left\vert L^\psi_{n}\left( x-t\right) \right\vert d\mu \left(t\right)  \notag \\
&\leq &\frac{cM_{N}^{1/p}}{l_n}\underset{I_{N}}{\int }
\text{ }\underset{j=M_{N}+1}{\overset{n-1}{\sum }}\frac{\left\vert
K^\psi_{j}\left( x-t\right) \right\vert }{j+1}d\mu \left( t\right)  \notag \\
&+&\frac{cM_N^{1/p}}{l_n}\int_{I_N}\left\vert
K^\psi_n\left(x-t\right)\right\vert d\mu\left(t\right).  \notag
\end{eqnarray}

Let $x\in I_{N}\left( x_{k}e_{k}+x_{l}e_{l}\right) ,$ $\ 1\leq x_{k}\leq
m_{k}-1,$ $1\leq x_{l}\leq m_{l}-1,$ $k=0,...,N-2,$ $l=k+1,...,N-1.$ From
Lemmas \ref{l1} and \ref{l2} it follows that
\begin{equation} \label{29}
\left| R^\psi_{n}a\left( x\right)\right| \leq\frac{cM_{l}M_{k}M_{N}^{1/p-2}}{\log(n+1)}.  
\end{equation}

Let $x\in I_{N}\left( x_{k}e_{k}\right) ,$ $1\leq x_{k}\leq m_{k}-1,$ $%
k=0,...,N-1.$ Applying Lemmas \ref{l1} and \ref{l2} we can conclude that
\begin{equation} \label{32}
\left\vert R^\psi_{n}a\left(x\right)\right\vert\leq{M_{N}^{1/p-1}M_{k}}.  
\end{equation}

By combining (\ref{3}) and (\ref{26}-\ref{32}) we obtain that
\begin{eqnarray} \label{7aaa}
&&\int_{\overline{I_{N}}}\left\vert R^\psi_na\left( x\right) \right\vert
^{p}d\mu \left( x\right) 
\end{eqnarray}   
\begin{eqnarray*}
&=&\overset{N-2}{\underset{k=0}{\sum}}\overset{N-1}{\underset{l=k+1}{\sum}}\sum\limits_{x_j=0,j\in\{l+1,...,N-1}^{m_{j-1}}\int_{I_N^{k,l}}\left\vert R^\psi_na \right\vert^{p}d\mu+\overset{N-1}{\underset{k=0}{\sum}} \int_{I_N^{k,N}}\left\vert R^\psi_na\right\vert^pd\mu   \\ \notag
&\leq&c\overset{N-2}{\underset{k=0}{\sum }}\overset{N-1} {\underset{l=k+1}{\sum }}\frac{m_{l+1}\dots m_{N-1}}{M_{N}}\frac{\left( M_{l}M_{k}\right)^{p}M_{N}^{1-2p}}{{\log}^{p}(n+1)}+\overset{N-1}{\underset{k=0}{\sum }}\frac{1}{M_{N}}M_{k}^{p}M_{N}^{1-p}  \\
&\leq&\frac{cM_N^{1-2p}}{{\log^p(n+1)}}\overset{N-2}{\underset{k=0}{\sum}}\overset{N-1}{\underset{l=k+1}{\sum}}\frac{\left(M_lM_k\right)^p}{M_l}+\overset{N-1}{\underset{k=0}{\sum}}\frac{M_{k}^{p}}{M_{N}^p}  \notag \\
&=&\frac{cM_N^{1-2p}}{\log^p(n+1)}\overset{N-2}{\underset{k=0}{\sum}}\frac{1}{M_k^{1-2p}}\overset{N-1}{\underset{l=k+1}{\sum}}\frac{M_k^{1-p}}{M_l^{1-p}}+\overset{N-1}{\underset{k=0}{\sum}}\frac{M_k^p}{M_{N}^p} \notag \\ &\leq&\frac{cM_N^{1-2p}}{\log^p(n+1)}+c_p.  \notag
\end{eqnarray*}

It is easy to see that
\begin{equation} \label{6aaa}
\overset{\infty}{\underset{n=M_{N}+1}{\sum }}\frac{1}{n^{2-2p}}\leq \frac{c}{M_{N}^{1-2p}},\text{ for }0<p< 1/2.
\end{equation}

By combining (\ref{7aaa}) and (\ref{6aaa}) we get that
\begin{eqnarray*}
&&\overset{\infty}{\underset{n=M_N+1}{\sum}}\frac{\log^p n\int_{\overline{I_N}}\left\vert R_na\right\vert ^{p}d\mu}{n^{2-2p}}  \\
&\leq &\overset{\infty}{\underset{n=M_N+1}{\sum}}\left( \frac{c_pM_N^{1-2p}}{n^{2-p}}+\frac{c_p}{n^{2-p}}\right)+c_p \\
&\leq&c_pM_N^{1-2p}\overset{\infty}{\underset{n=M_{N}+1}{\sum}}\frac{1}{n^{2-2p}}+\overset{\infty}{\underset{n=M_N+1}{\sum }}\frac{1}{n^{2-p}}+c_p\leq C_p<\infty.
\end{eqnarray*}
It means that (\ref{reiszbounded}) holds true and the proof is complete.
\end{proof}

Our next main result shows in particular that the inequality in Theorem \ref{threisz_2} is in a special sense sharp at least in the case of Walsh-Fourier series (c.f. also problem 2 in the next section).
\begin{theorem} \label{reisz_negative_th_2}
Let $0<p<1/2$ and $\Phi :\mathbb{N}\rightarrow \lbrack 1,\infty )$ be any non-decreasing function, satisfying the condition
\begin{equation}\label{dir222}
\underset{n\rightarrow \infty }{\lim }\Phi \left( n\right) =+\infty .
\end{equation}
Then there exists a martingale $f\in H_p\left(G_2\right) $ such that
\begin{equation}\label{theoremjemala}
\sum_{n=1}^{\infty}\frac{\log^p n
\left\Vert R^w_nf\right\Vert_p^p\Phi\left(n\right)}{n^{2-2p}} =\infty,
\end{equation}
where $R^w_{n}f$ denotes the $n$-th Reisz logarithmic means with respect to Walsh-Fourier series of
$f.$ 
\end{theorem}

\begin{proof}
It is evident that if we assume that $\Phi\left( n\right)\geq cn,$ where $c$ is some positive constant then 
\begin{equation*}
\frac{\log^p n\Phi\left(n\right)}{n^{2-2p}}\geq n^{1-2p}\log^p n \rightarrow \infty, \ \text{as} \ n\rightarrow \infty,
\end{equation*}
and also \eqref{theoremjemala} holds. So, without lost the generality we may assume that 	there exists an increasing sequence of positive integers
$\left\{ \alpha^{\prime} _{k}:k\in \mathbb{N} \right\} $ such that
\begin{equation} \label{3jemala}
\Phi\left( \alpha^{\prime} _{k}\right)=o(\alpha^{\prime} _{k}), \ \text{as} \ k\rightarrow \infty.
\end{equation}
	
Let $\left\{ \alpha _{k}:k\in \mathbb{N} \right\}\subseteq \left\{ \alpha^{\prime} _{k}:k\in \mathbb{N} \right\}$ be an increasing sequence of positive integers such that $\alpha _{0}\geq 2$ and
\begin{equation}
\sum_{k=0}^{\infty }\frac{1}{\Phi^{1/2}(2^{2\alpha_k})}<\infty,  \label{3a}
\end{equation}
	
\begin{equation} \label{4}
\sum_{\eta =0}^{k-1}\frac{2^{2\alpha _{\eta }/p}} {\Phi^{1/2p}(2^{2\alpha_\eta})}\leq \frac{2^{2\alpha _{k-1}/p+1}}{\Phi^{1/2p}(2^{2\alpha_{k-1}})}, 
\end{equation}
	
\begin{equation} \label{5}
\frac{2^{2\alpha _{k-1}/p+1}}{\Phi^{1/2p}(2^{2\alpha_{k-1}})}\leq \frac{1}{128\alpha_k}\frac{2^{2\alpha _k(1/p-2)}}{\Phi^{1/2p}(2^{2\alpha_k})}.  
\end{equation}
	
We note that under condition \eqref{3jemala} we can conclude that
$$\frac{2^{2\alpha _{\eta }/p}} {\Phi^{1/2p}(2^{2\alpha_\eta})}
\geq
{\left(\frac{2^{2\alpha _{\eta }}} {\Phi(2^{2\alpha_\eta})}\right)}^{1/2p} \rightarrow \infty, \ \text{as} \ \eta\rightarrow \infty$$
and it immediately follows that such an increasing sequence $\left\{ \alpha_k:k\in \mathbb{N}\right\} ,$ which satisfies conditions (\ref{3a})-(\ref{5}), can be constructed.
	
Let 
\begin{equation*}
f^{\left( A\right) }\left( x\right) :=\sum_{\left\{ k;\text{ }2\alpha_{k}<A\right\} }\lambda _{k}a_{k},
\end{equation*}
where 
\begin{equation*}
\lambda_k=\frac{1}{\Phi^{1/2p}(2^{2\alpha_k})}
\end{equation*}
and
\begin{equation*}
a_{k}={2^{2\alpha _{k}(1/p-1)}}\left( D_{2^{2\alpha _{k}+1}}-D_{2^{2\alpha_{k}}}\right) .
\end{equation*}
	
From (\ref{3a}) and Lemma \ref{atomicdec} we can conclude that $f=\left( f^{\left(n\right) },n\in \mathbb{N} \right) \in H_p(G_2).$
	
It is easy to show that
\begin{equation} \label{8}
\widehat{f}^w(j)=\left\{ 
\begin{array}{l}
\frac{2^{2\alpha _{k}(1/p-1)}}{\Phi^{1/2p}(2^{2\alpha_k})},\,\,\text{ if \thinspace\thinspace }j\in \left\{ 2^{2\alpha _{k}},...,\text{ ~}2^{2\alpha_{k}+1}-1\right\} ,\text{ }k\in \mathbb{N}, \\ 
0,\text{ \thinspace \thinspace \thinspace if \thinspace \thinspace
\thinspace }j\notin \bigcup\limits_{k=1}^{\infty }\left\{ 2^{2\alpha
_k},...,\text{ ~}2^{2\alpha _{k}+1}-1\right\}.
\end{array}\right.  
	\end{equation}
	
For $n=\sum_{i=1}^s 2^{n_i},$ $n_1<n_2<...<n_s$ we denote
\begin{equation*}
\mathbb{A}_{0,2}:=\left\{ n\in \mathbb{N}:\text{ }n=2^0+2^2+
\sum_{i=3}^{s_n}2^{n_i}\right\}.
\end{equation*}
	
Let $2^{2\alpha _{k}}\leq j\leq 2^{2\alpha_{k}+1}-1$ and $j\in \mathbb{A}_{0,2}.$ Then 
\begin{eqnarray} \label{10a}
R^w_jf &=&\frac{1}{l_j}\sum_{n=1}^{2^{2\alpha_k}-1}\frac{S_nf}{n} +\frac{1}{l_j}\sum_{n=2^{2\alpha_k}}^{j}\frac{S_{n}f}{n}:=I+II. 
\end{eqnarray}
	
Let $n<2^{2\alpha _{k}}.$ Then from (\ref{4}), (\ref{5}) and (\ref{8}) we have that
\begin{eqnarray*}
\left\vert S^w_{n}f\left( x\right) \right\vert 
&\leq &\sum_{\eta =0}^{k-1}\sum_{v=2^{2\alpha _{\eta }}}^{2^{2\alpha _{\eta}+1}-1}\left\vert \widehat{f}^w(v)\right\vert \leq \sum_{\eta
=0}^{k-1}\sum_{v=2^{2\alpha _{\eta }}}^{2^{2\alpha _{\eta }+1}-1}\frac{%
2^{2\alpha _{\eta }(1/p-1)}}{\Phi^{1/2p}(2^{2\alpha_\eta})} \\
&\leq &\sum_{\eta =0}^{k-1}\frac{2^{2\alpha _{\eta }/p}} {\Phi^{1/2p}(2^{2\alpha_\eta})}\leq \frac{2^{2\alpha _{k-1}/p+1}}{\Phi^{1/2p}(2^{2\alpha_{k-1}})}\leq \frac{1}{128\alpha_k}\frac{2^{2\alpha _k(1/p-2)}}{\Phi^{1/2p}(2^{2\alpha_k})}.
\end{eqnarray*}
	
Consequently,	
\begin{eqnarray} \label{11a}
\left\vert I\right\vert &\leq &\frac{1}{l_j}\underset{n=1}
{\overset{2^{2\alpha_k}-1}{\sum}}\frac{\left\vert S^w_n f\left( x\right)\right\vert}{n} \\
&\leq&\frac{1}{l_{2^{2\alpha_k}}}\frac{1}{128\alpha_k}\frac{2^{2\alpha _k(1/p-2)}}{\Phi^{1/2p}(2^{2\alpha_k})}
\sum_{n=1}^{2^{2\alpha_k}-1} \frac{1}{n}\leq\frac{1}{128\alpha_k}\frac{2^{2\alpha _k(1/p-2)}}{\Phi^{1/2p}(2^{2\alpha_k})}.  \notag
\end{eqnarray}

Let $2^{2\alpha _{k}}\leq n\leq 2^{2\alpha_{k}+1}-1.$ Then we have the
following
\begin{eqnarray*}
S^w_{n}f &=&\sum_{\eta =0}^{k-1}\sum_{v=2^{2\alpha _{\eta }}}^{2^{2\alpha
_{\eta }+1}-1}\widehat{f}^w(v)w_{v}+\sum_{v=2^{2\alpha _{k}}}^{n-1}\widehat{f}^w(v)w_{v} \\
&=&\sum_{\eta=0}^{k-1}\frac{2^{{2\alpha _{\eta }}\left(1/p-1\right)}} {\Phi^{1/2p}(2^{2\alpha_\eta})}\left( D^w_{2^{2\alpha _{\eta}+1}} -D^w_{2^{2\alpha_{\eta}}}\right)+\frac{2^{{2\alpha_k}\left( 1/p-1\right)}} {\Phi^{1/2p}(2^{2\alpha_k})}\left(D^w_{n}-D^w_{2^{2\alpha_k}}\right) .
\end{eqnarray*}
	
This gives that
	
\begin{equation}\label{12a}
II=\frac{1}{l_j}\underset{n=2^{2\alpha_k}}{\overset{2^{2\alpha _{k}+1}}{\sum }}\ \frac{1}{n}\left(\sum_{\eta=0}^{k-1}\frac{2^{2\alpha _{\eta}\left(1/p-1\right)}}{\Phi^{1/2p}(2^{2\alpha_\eta})}\left( D^w_{2^{2\alpha_{\eta}+1}}-D^w_{2^{2\alpha _{\eta }}}\right)\right)  
\end{equation}%
\begin{equation*}
+\frac{1}{l_{j}}\frac{2^{2\alpha_k\left(1/p-1\right)}}{\Phi^{1/2p}(2^{2\alpha_k})}\sum_{n=2^{2\alpha_k}}^{j}\frac{\left( D^w_n-D^w_{2^{2\alpha _k}}\right) }{n}:=II_{1}+II_{2}.
\end{equation*}
	
Let $x\in I_{2}(e_{0}+e_{1})\in I_{0}\backslash I_{1}.$ According to well-known equalities for Dirichlet kernels (for details see e.g. \cite{G-E-S} and \cite{S-W-S}):
recall that
\begin{equation} \label{1dn}
D^w_{2^{n}}\left( x\right)=\left\{ 
\begin{array}{ll}
2^{n}, & \,\text{if\thinspace \thinspace \thinspace }x\in I_{n} \\ 
0, & \text{if}\ \ \ x\notin I_{n}
\end{array}
\right.  
\end{equation}
and
\begin{equation}\label{2dn}
D^w_n=w_n\overset{\infty }{\underset{k=0}{\sum }}n_kr_kD^w_{2^k}=w_n
\overset{\infty }{\underset{k=0}{\sum }}n_{k}\left(
D^w_{2^{k+1}}-D^w_{2^{k}}\right),\text{ for \ }n=\overset{\infty }{\underset{i=0}{\sum }}n_{i}2^{i},
\end{equation}
so we can conclude that
\begin{equation*}
D^w_n\left( x\right) =\left\{ 
\begin{array}{ll}
w_{n}, & \,\text{if\thinspace \thinspace \thinspace }n\ \ \text{is odd
number,} \\ 
0, & \text{if}\,\,n\ \ \text{is even number.}
\end{array}%
\right.
\end{equation*}
	
Since $\alpha _{0}\geq 2,\ \ k\in \mathbb{N}$ we obtain that $2\alpha_{k}\geq 4,$ for all $k\in \mathbb{N}$ and if we apply (\ref{1dn}) we get that
\begin{equation}
II_{1}=0  \label{13a}
\end{equation}
and 
\begin{eqnarray*}
&&II_2=\frac{1}{l_j}\frac{2^{2\alpha _k(1/p-1)}}{\Phi^{1/2p}(2^{2\alpha_k})}\sum_{n=2^{2\alpha _k-1}}^{(j-1)/2}\frac{w_{2n+1}}{2n+1}=\frac{1}{l_j}\frac{2^{2\alpha _k(1/p-1)}r_{1}}{\Phi^{1/2p}(2^{2\alpha_k})}\sum_{n=2^{2\alpha _k-1}}^{(j-1)/2}\frac{w_{2n}}{2n+1}.
\end{eqnarray*}
	
Let $x\in I_{2}(e_{0}+e_{1}).$ Then, by the definition of Walsh functions, we get that
\begin{equation*}
w_{4n+2}=r_{1}w_{4n}=-w_{4n}
\end{equation*}
and
\begin{eqnarray} \label{14a}
&&\left\vert II_2\right\vert=\frac{1}{l_j}\frac{2^{2\alpha _k(1/p-1)}}{\Phi^{1/2p}(2^{2\alpha_k})}\left\vert \sum_{n=2^{2\alpha _k-1}}^{(j-1)/2}\frac{w_{2n}}{2n+1}\right\vert 
\end{eqnarray}
\begin{eqnarray*}
&=&\frac{1}{l_j}\frac{2^{2\alpha_k(1/p-1)}}{\Phi^{1/2p}(2^{2\alpha_k})}\left\vert\frac{w_{j-1}}{j}+\sum_{n=2^{2\alpha_k-2}+1}^{(j-1)/4}\left(\frac{w_{4n-4}}{4n-3}+\frac{w_{4n-2}}{4n-1}\right)\right\vert  \\ \notag
&=&\frac{1}{l_j}\frac{2^{2\alpha _{k}(1/p-1)}}{\Phi^{1/2p}(2^{2\alpha_k})}\left\vert \frac{w_{j-1}}{j}+\sum_{n=2^{2\alpha_{k}-2}+1}^{(j-1)/4}\left( \frac{w_{4n-4}}{4n-3}-\frac{w_{4n-2}}{4n-1}\right) \right\vert  \\ \notag
&\geq&\frac{c}{\log({2^{2\alpha_k+1}})}\frac{2^{2\alpha_k(1/p-1)}}{\Phi^{1/2p}(2^{2\alpha_k})}\left( \left\vert\frac{w_{j-1}}{j}\right\vert-\sum_{n=2^{2\alpha _{k}-2}+1}^{(j-1)/4} \left\vert{w_{4n-4}}\right\vert\left( \frac{1}{4n-3}-\frac{1}{4n-1}\right) \right)  \\ \notag
&&\geq\frac{1}{4\alpha_k}\frac{2^{2\alpha_k(1/p-1)}}{\Phi^{1/2p}(2^{2\alpha_k})}\left( \frac{1}{j}-\sum_{n=2^{2\alpha_k-2}+1}^{(j-1)/4}\left( \frac{1}{4n-3}-\frac{1}{4n-1}\right)\right).
\end{eqnarray*}

By simple calculation we can conclude that
\begin{eqnarray*}
&&\sum_{n=2^{2\alpha_k-2}+1}^{(j-1)/4}\left(\frac{1}{4n-3}- \frac{1} {4n-1}\right)=\sum_{n=2^{2\alpha_k-2}+1}^{(j-1)/4}\frac{2}{(4n-3)(4n-1)} \\
&\leq& \sum_{n=2^{2\alpha _k-2}+1}^{(j-1)/4} \frac{2}{(4n-4)(4n-2)}=\frac{1}{2}\sum_{n=2^{2\alpha _k-2}+1}^{(j-1)/4} \frac{1}{(2n-2)(2n-1)} \\
&\leq&\frac{1}{2}\sum_{n=2^{2\alpha _k-2}+1}^{(j-1)/4} \frac{1}{(2n-2)(2n-2)}=\frac{1}{8}\sum_{n=2^{2\alpha _k-2}+1}^{(j-1)/4} \frac{1}{(n-1)(n-1)} \\
&\leq& \frac{1}{8}\sum_{n=2^{2\alpha _k-2}+1}^{(j-1)/4} \frac{1}{(n-1)(n-2)}=\frac{1}{8}\sum_{l=2^{2\alpha _k-2}+1}^{(j-1)/4} \left(\frac{1}{n-2}-\frac{1}{n-1} \right)\\
&\leq &\frac{1}{8}\left(\frac{1}{2^{2\alpha_k-2}-1}-\frac{4}{j-5}\right)
\leq \frac{1}{8}\left(\frac{1}{2^{2\alpha_k-2}-1}-\frac{4}{j}\right).
\end{eqnarray*}

Since $2^{2\alpha _{k}}\leq j\leq 2^{2\alpha_{k}+1}-1,$ where $\alpha _{k}\geq 2,$ we obtain that 
$$\frac{2}{2^{2\alpha_k}-4}\leq \frac{2}{2^{4}-4}=\frac{1}{6}$$
and
\begin{eqnarray} \label{14aR}
\left\vert II_2\right\vert&\geq&\frac{1}{4\alpha_k}\frac{2^{2\alpha _k(1/p-1)}}{\Phi^{1/2p}(2^{2\alpha_k})}\left( \frac{1}{j}-\frac{1}{8}\left(\frac{1}{2^{2\alpha_k-2}-1}-\frac{4}{j}
\right)\right) \\ \notag
&\geq &\frac{1}{4\alpha_k}\frac{2^{2\alpha _k(1/p-1)}}{\Phi^{1/2p}(2^{2\alpha_k})}\left( \frac{3}{2j}-\frac{1}{2^{2\alpha_k+1}-8}\right) \\ \notag
&\geq &\frac{1}{4\alpha_k}\frac{2^{2\alpha _k(1/p-1)}}{\Phi^{1/2p}(2^{2\alpha_k})}\left( \frac{3}{4}\frac{1}{2^{2\alpha_{k}}}-\frac{1}{2}\frac{1}{2^{2\alpha_k}-4}\right) \\ \notag
&\geq &\frac{1}{4\alpha_k}\frac{2^{2\alpha _k(1/p-1)}}{\Phi^{1/2p}(2^{2\alpha_k})}\left( \frac{1}{4}\frac{1}{2^{2\alpha_{k}}}+\frac{1}{2}\frac{1}{2^{2\alpha_{k}}}-\frac{1}{2}\frac{1}{2^{2\alpha_k}-4}\right) \\ \notag
&=&\frac{1}{4\alpha_k}\frac{2^{2\alpha _k(1/p-1)}}{\Phi^{1/2p}(2^{2\alpha_k})}\left( \frac{1}{4}\frac{1}{2^{2\alpha_{k}}}-\frac{2}{2^{2\alpha_k}(2^{2\alpha_k}-4)}\right)
\\ \notag
&\geq &\frac{1}{4\alpha_k}\frac{2^{2\alpha _k(1/p-1)}}{\Phi^{1/2p}(2^{2\alpha_k})}\left( \frac{1}{4}\frac{1}{2^{2\alpha_{k}}}-\frac{1}{6}\frac{1}{2^{2\alpha_{k}}}\right)
\\ \notag
&\geq &\frac{1}{48\alpha_k}\frac{2^{2\alpha _k(1/p-2)}}{\Phi^{1/2p}(2^{2\alpha_k})}
\geq \frac{1}{64\alpha_k}\frac{2^{2\alpha _k(1/p-2)}}{\Phi^{1/2p}(2^{2\alpha_k})}.
\end{eqnarray}
	
By combining (\ref{5}), (\ref{10a})-(\ref{14aR}) for $\in I_{2}(e_{0}+e_{1})$ and $0<p<1/2$ we find that

\begin{eqnarray*}
&&\left\vert R^w_j f\left( x\right) \right\vert \geq
\left\vert II_{2}\right\vert-\left\vert II_{1}\right\vert -\left\vert I \right\vert\\
&\geq &\frac{1}{64\alpha_k}\frac{2^{2\alpha _k(1/p-2)}}{\Phi^{1/2p}(2^{2\alpha_k})}-\frac{1}{128\alpha_k}\frac{2^{2\alpha _k(1/p-2)}}{\Phi^{1/2p}(2^{2\alpha_k})}
=\frac{1}{128\alpha_k}\frac{2^{2\alpha _k(1/p-2)}}{\Phi^{1/2p}(2^{2\alpha_k})}.
\end{eqnarray*}
	
Hence,
\begin{eqnarray} \label{16}
&&\left\Vert R^w_{j}f\right\Vert_{weak-L_p(G_2)}^p  \\ \notag
&\geq &\frac{1}{128\alpha_k^p}\frac{2^{2\alpha _k(1-2p)}}{\Phi^{1/2}(2^{2\alpha_k})}\mu\left\{ x\in G_2:\left\vert R^w_{j}f\right\vert \geq \frac{1}{128\alpha_k}\frac{2^{2\alpha _k(1/p-2)}}{\Phi^{1/2p}(2^{2\alpha_k})}\right\} ^{1/p} \\ \notag
&\geq &\frac{1}{128\alpha_k^p}\frac{2^{2\alpha _k(1-2p)}}{\Phi^{1/2}(2^{2\alpha_k})}\mu \left\{ x\in I_{2}(e_{0}+e_{1}):\left\vert
R^w_{j}f\right\vert \geq \frac{1}{128\alpha_k}\frac{2^{2\alpha _k(1/p-2)}}{\Phi^{1/2p}(2^{2\alpha_k})}\right\} \\ \notag
&\geq &\frac{1}{128\alpha_k^p}\frac{2^{2\alpha _k(1-2p)}}{\Phi^{1/2}(2^{2\alpha_k})}(\mu \left( x\in I_{2}(e_{0}+e_{1})\right))>\frac{1}{516\alpha_k^p}\frac{2^{2\alpha _k(1-2p)}}{\Phi^{1/2}(2^{2\alpha_k})}.
\end{eqnarray}
Moreover,

\begin{eqnarray*}
&&\underset{j=1}{\overset{\infty }{\sum }}\frac{\left\Vert R^w
_{j}f\right\Vert _{weak-L_p(G_2)}^{p}\log^p{(j)}\Phi(j)}{j^{2-2p} } 
\end{eqnarray*}
\begin{eqnarray*}
&\geq& \underset{\left\{ j\in \mathbb{A}_{0,2}:\text{ }2^{2\alpha_k}< j\leq 2^{2\alpha_{k}+1}-1\right\} }{\sum }\frac{\left\Vert R^w
_{j}f\right\Vert _{weak-L_p}^{p}\log^p{(j)}\Phi(j)}{j^{2-2p} } \\
&\geq &\frac{c}{\alpha_k^p}\frac{2^{2\alpha _k(1-2p)}}{\Phi^{p/2}(2^{2\alpha_k)}}\underset{\left\{ j\in \mathbb{A}_{0,2}:\text{ }2^{2\alpha_k}< j\leq 2^{2\alpha_{k}+1}-1\right\} }{\sum }\frac{\log^p{(j)}\Phi(j)}{j^{2-2p}}\\
&\geq &\frac{c\Phi(2^{2\alpha_{k}})\log^p{(2^{2\alpha_k})}}{\alpha_k^p}\frac{2^{2\alpha _k(1-2p)}}{\Phi^{1/2}(2^{2\alpha_k})}\underset{\left\{ j\in \mathbb{A}_{0,2}:\text{ }2^{2\alpha_k}< j\leq 2^{2\alpha_{k}+1}-1\right\} }{\sum }\frac{1}{j^{2-2p}}\\ 
&\geq& \Phi^{1/2}(2^{2\alpha_{k}})\rightarrow \infty ,\text{ \ as \ \ }k\rightarrow \infty.
\end{eqnarray*}
	
	The proof is complete.
\end{proof}

\section{\textbf{FINAL REMARKS AND OPEN PROBLEMS}}

In this section we present some final remarks and open problems, which can be interesting for the researchers, who work in this area. The first problem reads:
\begin{problem} For any $f\in {H_{1/2}},$ is it possible to find strong convergence theorems for Reisz means $R^w_m,$ where $\alpha =w\text{ or }\alpha=\psi$?
\end{problem} 
\begin{remark} Similar problems for Fej\'er means with respect to Walsh and Vilenkin systems can be found in \cite{BPTW}, \cite{B-T}, \cite{tep11} (see also \cite{tut1} and \cite{We}). Our method and estimations of Reisz and Fej\'er kernels (see Lemmas 1 and 2) do not give possibility to prove even similar strong convergence result as for the case of Fejer means. In particular, for any $f\in H_{1/2},$ is it possible to prove the following inequality:
\begin{equation*}
\frac{1}{\log n}\overset{n}{\underset{k=1}{\sum}} \frac{\left\Vert R^\alpha_k f\right\Vert_{1/2}^{1/2}}{k}\leq
c\left\Vert f\right\Vert_{H_{1/2}}^{1/2},  \ \text{where} \ \alpha =w\text{ or }\alpha =\psi?
\end{equation*}
\end{remark} 
It is interesting to generalize Theorem \ref{reisz_negative_th_2} for Vilenkin systems:
\begin{problem}
For $0<p<1/2$ and any non-decreasing function $\Phi :\mathbb{N}\rightarrow \lbrack 1,\infty )$ satisfying the conditions 
\begin{equation*}
\underset{n\rightarrow \infty }{\lim }\Phi \left( n\right) =+\infty,
\end{equation*}
is it possible to find a martingale $f\in H_p\left(G_m\right) $ such that
\begin{equation*}
\sum_{n=1}^{\infty}\frac{\log^p n \left\Vert R^\psi_nf\right\Vert_p^p\Phi\left(n\right)}{n^{2-2p}} =\infty.
\end{equation*}
where $R^\psi_{n}f$ denotes the $n$-th Reisz logarithmic means with respect to Vilenkin-Fourier series of $f.$ 
\end{problem}

\begin{problem} Is it possible to find a martingale $f\in {H_{1/2}},$ such that 
$$\sup_{n\in \mathbb{N}} \left\Vert R^\alpha_n f\right\Vert_{1/2}=\infty,$$
where $\alpha =w\text{ or }\alpha =\psi$?
\end{problem} 

\begin{remark}
For $0<p<1/2,$ divergence in the space $L_p$ of Reisz logarithmic means with respect to Walsh and Vilenkin systems of martingale $f\in H_p$  was already proved in \cite{PTW1}.
\end{remark}

\begin{problem}
For any $f\in {H_p}$ \ $\left( 0<p\le 1/2 \right),$ is it possible to find necessary and sufficient conditions for the indexes 
$k_j$ for which 
\[\begin{matrix}
{\left\|R^\alpha_{k_j} f-f \right\|}_{{H_p}}\to 0, & as & j\to \infty,  \\
\end{matrix}\]
where $\alpha =w\text{ or }\alpha =\psi$?
\end{problem} 
\begin{remark}
Similar problem for partial sums and Fejer means with respect to Walsh and Vilenkin systems can be found in Tephnadze \cite{tep8}, \cite{tep9} and \cite{tep10}.
\end{remark}

\begin{problem}
Is it possible to find necessary and sufficient conditions in terms of the one-dimensional modulus of continuity of martingale 
$f\in {H_p}$ \ $\left( 0<p\le 1/2 \right),$ for which 
\[\begin{matrix}
{\left\|R^\alpha_j f-f \right\|}_{{H_p}}\to 0, & as & j\to \infty,  \\
\end{matrix}\]
where $\alpha =w\text{ or }\psi$?
\end{problem} 
\begin{remark}
Approximation properties of some summability methods in the classical and real Hardy spaces were considered by Oswald \cite{Os}, Kryakin and Trebels \cite{KT}, Storoienko \cite{St1}, \cite{St2} and for martingale Hardy spaces in Fridli, Manchanda and Siddiqi \cite{FMS} (see also \cite{4}, \cite{FR}),  Nagy \cite{nagy}, \cite{na}, \cite{n}, Tephnadze \cite{tep8}, \cite{tep9}, \cite{tep10}.
\end{remark}

\end{document}